\documentclass[12pt]{amsart}
\usepackage{amsmath}
\usepackage[active]{srcltx}
\usepackage{t1enc}
\usepackage[latin2]{inputenc}
\usepackage{verbatim}
\usepackage{amsmath,amsfonts,amssymb,amsthm}
\usepackage[mathcal]{eucal}
\usepackage{enumerate}
\usepackage[centertags]{amsmath}
\usepackage{graphics}
\numberwithin{equation}{section}

\newtheorem{theorem}{Theorem}

\newtheorem{OldTheorem}{Theorem}

\newtheorem{lemma}{Lemma}

\newtheorem{problem}{Problem}

\def\dist{{\rm dist}}
\def\sign{{\rm sign\,}}
\def\supp{{\rm supp\,}}

\def\H{\rm H}

\def\ZR{\ensuremath{\mathbb R}}
\def\ZZ{\ensuremath{\mathbb Z}}

\def\ZT{\ensuremath{\mathbb T}}
\def\ZI{\ensuremath{\mathbb I}}

\def\md#1#2\end{equation*}{\ifx0#1
\begin{equation*} #2 \end{equation*}\fi  
\ifx1#1\begin{equation}#2\end{equation}\fi   
\ifx2#1\begin{align*}#2\end{align*}\fi   
\ifx3#1\begin{align}#2\end{align}\fi    
\ifx4#1\begin{gather*}#2\end{gather*}\fi  
\ifx5#1\begin{gather}#2\end{gather}\fi   
\ifx6#1\begin{multline*}#2\end{multline*}\fi  
\ifx7#1\begin{multline}#2\end{multline}\fi  
\ifx8#1$$\begin{array}#2\end{array}$$\fi  
\ifx9#1\begin{eqnarray}#2\end{eqnarray}\fi  
}
\newcommand {\e }[1]{(\ref{#1})}
\newcommand {\lem }[1]{Lemma \ref{#1}}
\newcommand {\trm }[1]{Theorem \ref{#1}}

\begin{document}
\title{On exceptional sets of Hilbert transform}

\author
{G. A. Karagulyan}

\address{G. A. Karagulyan, Faculty of Mathematics and Mechanics, Yerevan State University, Alex Manoogian, 1, 0025, Yerevan, Armenia}
\email{g.karagulyan@ysu.am}

\subjclass[2010]{42B20}%
\keywords{Hilbert transform, exceptional null set, divergent Fourier series}
\date{}
\maketitle
\begin{abstract} We prove several theorems concerning the exceptional sets of Hilbert transform on the real line. In particular, it is proved that any null set is exceptional set for the Hibert transform of an indicator function. The paper also provides a real variable approach to the Kahane-Katsnelson theorem on divergence of Fourier series. 
\end{abstract}
\maketitle
\section{Introduction}
The study of exceptional sets is common in Harmonic Analysis and some related fields. One century ago Lusin \cite{Lus} proved that for any boundary null set $e$ (a set of measure zero) there exists a bounded analytic function on the unit disc, which has no radial limits at any point of $e$. This result was a significant complement to the theorem of Fatou, providing almost everywhere tangential convergence for the bounded analytic functions on the disc.  Kahane-Katznelson's \cite{KaKa} example of a continuous function, whose Fourier series diverges at any point of an arbitrary given null set was the counterpart of Carleson's \cite{Car} celebrated theorem on almost everywhere convergence of Fourier series. Some extensions of Kahane-Katznelson's theorem for Fourier series in different classical orthogonal systems the readers can find  in the papers \cite{ Bug,Buz,Gog,Khe,Luc,Pro,Ste,Tai}. 

It was discovered in the papers \cite{Kar1,Kar2} that such divergence phenomena is common for general sequences of bounded linear operators 
\begin{equation}\label{z12}
U_n:L^\infty(a,b) \to \text { bounded measurable functions on } (a,b)
\end{equation}
satisfying the localization property, that means for any function $f\in L^\infty(a,b)$ with $f(x)=1$, $x\in (\alpha,\beta)$, the sequence $U_nf(x)$ converges uniformly in $(\alpha,\beta)$. We denote by $\ZI_G$ the indicator function of a set $G\subset \ZR$. It was proved in \cite{Kar1} that
\begin{OldTheorem}[\cite {Kar1}]\label{OldT1}
	If the operator sequence \e{z12} satisfies the localization property, then for any null set $e\subset (a,b)$ there exists a measurable set $G\subset (a,b)$ such that $U_n\ZI_G(x)$ diverges at any $x\in e$.
\end{OldTheorem}
In \cite{Kar2} we obtained full characterization theorems for exceptional sets of general sequences of operators with localization property.

In this paper we consider the exceptional null set problem for the Hilbert transform. 
The Hilbert transform of a function $f\in L^1(\ZR)$ is the integral
\begin{equation*}
Hf(x)=\lim_{\varepsilon\to 0}H_\varepsilon f(x)=\lim_{\varepsilon\to 0}\frac{1}{\pi}\int_{|t-x|>\varepsilon}\frac{f(t)}{x-t}dt.
\end{equation*}
It is well-known the almost everywhere existence of this limit for the integrable functions (see for example \cite {Zyg}, ch. 4.3). The maximal Hilbert transform is defined by
\begin{equation*}
H^*f(x)=\sup_{\varepsilon>0}|H_\varepsilon f(x)|.
\end{equation*}
Examples of exceptional sets for the Hilbert transform have been only considered by Lusin in his famous book (\cite{Lus}, page 459). It was proved in \cite {Lus} the existence of an everywhere dense continuum null set $e\subset \ZR$, such that $H^*f(x)=\infty $ on $e$ for some $f\in C(\ZR)\cap L^1(\ZR)$. The following theorems shows that any null set $e$ can serve as an exceptional set for the Hilbert transform of some indicator function. Moreover, if $e$ is additionally compact, then instead of the indicator function it can be taken a continuous function. 
\begin{theorem}\label{T1}
	For any null set $e\subset \ZR$ there exists a set $E\subset \ZR$ of finite measure such that
	\begin{equation*}
	H^*\ZI_{E}(x)=\infty,\quad x\in e.
	\end{equation*}
\end{theorem} 
Note that \trm {T1} can not be deduced from \trm {OldT1}, since the operators $H_\varepsilon $ do not satisfy the localization property. Its proof as well as the proof of the next theorem are essentially based on characteristic properties of Hilbert transform. 
\begin{theorem}\label{T2}
	For any closed null set $e\subset \ZR$ there exists a continuous function $f\in C(\ZR)\cap L^1(\ZR)$ such that
	\begin{equation*}
	H^*f(x)=\infty,\quad x\in e.
	\end{equation*}
\end{theorem}
The proof of Kahane-Katznelson's theorem \cite{KaKa} uses methods of analytic functions. We will show in the last section that this theorem can be alternatively deduced from \lem {L1} below. 

The following questions are open.
\begin{problem}
	Is the statement of \trm {T2} valid for arbitrary null sets. 
\end{problem}
\begin{problem}\label{P2}
	Is the analogous of Kahane-Katznelson's theorem true for Walsh system (see for example \cite {Wade}).
\end{problem} 
Concerning to Problem 2 we note that Harris \cite{Har} has proved that for any compact null set $e\subset [0,1]$ there exists a continuous function, whose Walsh-Fourier series diverges at any $x\in e$.
\section{Intermediate results}

We say an open set $G\subset \ZR$ is of finite form (or finite-open), if it is a union of finite number of open intervals. For two measurable sets $E$ and $F$ we denote
\begin{equation*}
E\triangle F=(E\setminus F)\cup (F\setminus E)
\end{equation*}
and write $E\sim F$ in the case $|E\triangle F|=0$.
Let $E,E_n\subset $ be measurable sets. We write $F_n\Rightarrow F$ whenever we have
\begin{equation*}
|F_n\triangle F|\to 0 \text { as }n\to\infty.
\end{equation*}
For measurable functions $f$ and $f_n$, $n=1,2,\ldots $ the notation $f_n\Rightarrow f$ denotes a convergence in measure.

 The following theorem has own interest and it will be used in the proofs of the main theorems. 
\begin{theorem}\label{T3}
	Let 
	\begin{equation}\label{z1}
		\lambda>0,\quad \mu=\frac{1}{\pi}\ln (1-e^{-\lambda}).
	\end{equation}
	Then for an arbitrary measurable set $F\subset \ZR$ the sets
	\begin{align}
	&E=\left\{x\in \ZR\setminus F:\, H\ZI_{F}(x)<\mu\right\},\label{b38}\\
	&F^*=\{x\in \ZR\setminus E:\, \H\ZI_E(x)> \lambda \}\nonumber
	\end{align}
	satisfy the relations
	\begin{equation*}
	F\sim  F^*,\quad |E|=(e^\lambda-1)|F|.
	\end{equation*}
	Moreover, if $F$ is open (or finite-open), then we additionally have $F\subset  F^*$ (or $F=F^*$).
\end{theorem}
We shall often use the following property of the Hilbert transform. Namely, if for $f\in L^1(\ZR)$ vanishes on $(a,b)$, then $Hf(x)$ is decreasing on $(a,b)$. Consider numbers
\begin{equation}\label{b6}
c_k<a_k<b_k<c_{k+1},\quad k=0,1,2,\ldots,n+1
\end{equation}
where
\begin{equation*}
c_0=a_0=b_0=-\infty,\quad c_{n+1}=a_{n+1}=b_{n+1}=+\infty.
\end{equation*}
Denote
\begin{align}
&F=\bigcup_{k=1}^n(a_k,b_k),\label{b7}\\
&E=\bigcup_{k=1}^n(c_k,a_k).\label{b44}
\end{align}
Simple calculations shows that 
\begin{align}
H\ZI_F(x)=\frac{1}{\pi}\sum_{k=1}^n\int_{a_k}^{b_k}\frac{dt}{x-t}=\frac{1}{\pi}\sum_{k=1}^n\ln \left|\frac{x-a_k}{x-b_k}\right|,\label{b1}\\
H\ZI_E(x)=\frac{1}{\pi}\sum_{k=1}^n\int_{c_k}^{a_k}\frac{dt}{x-t}=\frac{1}{\pi}\sum_{k=1}^n\ln \left|\frac{x-c_k}{x-a_k}\right|
\end{align}
for any
\begin{equation}
x\in \ZR'=\ZR\setminus \{a_k,b_k:\,k=1,2,\ldots,n\}.
\end{equation}
Observe that the function \e{b1} is decreasing on each interval $(b_k,a_{k+1})$ and increasing on $(a_k,b_k)$, $k=0,1,\ldots,n$. Besides we have
\begin{align}
&\lim_{x\to a_k}H\ZI_F(x)=-\infty,\quad \lim_{x\to b_k}H\ZI_F(x)=+\infty,\quad k=1,2,\ldots,n,\label{z51}\\
&\lim_{x\to \pm\infty }H\ZI_F(x)=0.\label{z52}
\end{align}

The following lemma is the case of \trm {T3} when $F$ is finite-open.
\begin{lemma}\label{L1} Let the numbers $\lambda>0$, $\mu<0$ satisfy \e {z1}. 
If $F$ is an open set of the form \e{b7} with $b_{k-1}<a_k$, $k=2,3,\ldots n$, then the set
\begin{equation*}
E=\left\{x\in \ZR'\setminus F:\, H\ZI_{F}(x)<\mu\right\},\label{b5}
\end{equation*}
has the form \e{b44}, where $c_k$ satisfy the relation \e{b6} and we have
\begin{align}
&F=\{x\in \ZR'\setminus E:\, \H\ZI_E(x)> \lambda \},\label{b40}\\
&|E|=(e^\lambda-1)|F|.\label{b41}
\end{align}
\end{lemma}
\begin{proof}
From \e {z51}, \e {z52} and monotonicity of $H\ZI_F$ on $(b_{k-1},a_{k})$ it follows that given $\lambda>0$ uniquely determines numbers $c_k$ such that
\begin{equation*}
H\ZI_F(c_k)=\mu= \frac{1}{\pi}\ln (1-e^{-\lambda}),\,c_k\in (b_{k-1},a_{k}),\, k=1,2,\ldots,n,
\end{equation*}
and 
\begin{equation*}
E=\bigcup_{k=1}^n(c_k,a_k)=\left\{x\in \ZR'\setminus F:\, H\ZI_{F}(x)<\mu\right\}.
\end{equation*}
Equality \e{b1} implies that the numbers $c_k$ are the roots of the equation
\begin{equation*}\label{a4}
\frac{1}{\pi}\sum_{k=1}^n\ln \left|\frac{x-a_k}{x-b_k}\right|=\mu.
\end{equation*}
From \e{b6} it follows that $(c_j-b_k)(c_j-a_k)>0$ for any $j=1,2,\ldots, n$. Thus we conclude that $c_k$ are the roots of the algebraic equation
\begin{equation}\label{b9}
e^{\pi \mu} \prod_{k=1}^n(x-b_k)-\prod_{k=1}^n(x-a_k)=0,
\end{equation}
and according to B\'{e}zout's theorem we have
\begin{equation}\label{b45}
e^{\pi \mu} \prod_{k=1}^n(x-b_k)-\prod_{k=1}^n(x-a_k)=(e^{\pi \mu}-1)\prod_{k=1}^n(x-c_k).
\end{equation}
This implies that $b_k$ are the roots of the equation
\begin{equation*}
(e^{\pi \mu}-1)\prod_{k=1}^n(x-c_k)+\prod_{k=1}^n(x-a_k)=0
\end{equation*}
and therefore for
\begin{equation*}
\sum_{k=1}^n\ln \left|\frac{x-c_k}{x-a_k}\right|=-\ln(1-e^{\pi \mu})=\lambda.
\end{equation*}
That means
\begin{equation*}
H\ZI_E(b_k)=\lambda,\quad k=1,2,\ldots,n.
\end{equation*}
Thus, since $H\ZI_E$ is decreasing on $(a_k,c_{k+1})$ and $b_k\in (a_k,c_{k+1})$, we get \e {b40}.
Equality of the $x^{n-1}$ coefficients of the right and left sides of \e {b45} gives 
\begin{equation*}
\sum_{k=1}^nc_k=\frac{e^{\pi \mu}\sum_{k=1}^nb_k-\sum_{k=1}^na_k}{e^{\pi \mu}-1}.
\end{equation*}
This implies 
\begin{equation*}
|E|=\sum_{k=1}^n(a_k-c_k)=\frac{e^{\pi \mu}\sum_{k=1}^n(b_k-a_k)}{1-e^{\pi \mu}}=\frac{e^{\pi \mu}|F|}{1-e^{\pi \mu}}=(e^\lambda-1)|F|
\end{equation*}
and so we get \e {b41}.
\end{proof}

\begin{proof}[Proof of \trm {T3}]
Since $F$ is measurable, there exists a sequence of finite-open sets $F_n$ such that
\begin{equation}\label{b43}
F_n\Rightarrow F.
\end{equation}
Applying \lem{L1}, we may find finite-open sets $E_n$ such that
\begin{align}
&F_n=\{x\in \ZR'\setminus E_n:\, H\ZI_{E_n}(x)>\lambda\},\label{b10}\\
&E_n=\left\{x\in \ZR'\setminus F_n:\, H\ZI_{F_n}(x)<\mu\right\}\label{b11},\\
&|E_n|=(e^\lambda-1)|F_n|.
\end{align}
We have
\begin{equation}\label{b12}
\|H\ZI_{F_n}-H\ZI_F\|_2=\|\ZI_{F_n}-\ZI_F\|_2\to 0.
\end{equation}
Taking into account the monotonicity property of function $H\ZI_F(x)$, from \e{b11} and \e{b12} one can easily get
\begin{equation}\label{b13}
E_n\Rightarrow E,
\end{equation}
where $E$ is defined in \e{b38}. By the same reason \e {b13} implies $H\ZI_{E_n}(x)\Rightarrow H\ZI_E(x)$ and therefore we get the relation $F_n\Rightarrow F^*$, which together with \e{b43} gives us $F\sim  F^*$.
From \e{b12}, \e{b13} and relation \e {b41} between $E_n$ and $F_n$ we get
\begin{equation}\label{b42}
|E|=\lim_{n\to\infty} |E_n|=(e^\lambda-1)\lim_{n\to\infty} |F_n|=(e^\lambda-1)|F|.
\end{equation}
If $F$ is open, then
\begin{equation*}
F=\bigcup_{k=1}^\infty(a_k,b_k),
\end{equation*}
where the intervals $(a_k,b_k)$ are pairwise disjoint. Denote
\begin{equation*}
F_n=\bigcup_{k=1}^n(a_k,b_k).
\end{equation*}
Take an arbitrary point $x\in F$.  The points $x$ and  $x+\delta$ are in the same component interval $(a_{n(x)},b_{n(x)})$ for small enough $\delta>0$. So we have $x\in F_n$ for $n\ge n(x)$ and therefore by \e{b10} we conclude
\begin{equation}\label{b39}
H\ZI_{E_n}(x+\delta)>\lambda,\quad n>n(x).
\end{equation}
Since  $H\ZI_{E}(x)$ is decreasing in $(a_{n(x)},b_{n(x)})$, from \e{b39} we get
\begin{equation*}
H\ZI_{E}(x)>H\ZI_{E}(x+\delta)=\lim_{n\to\infty}H\ZI_{E_n}(x+\delta)\ge\lambda.
\end{equation*}
This implies $x\in F^*$ and therefore we get $F\subset F^*$.
\end{proof}

\section{Proofs of main theorems}
Let $G\subset \ZR$ be an open set.  To any component interval $(a,b)$ of $G$ we associate the intervals
\begin{equation*}
\left[a+\frac{ b-a}{2^{j+1}},a+ \frac{b-a}{2^{j}} \right),\,\left[b-\frac{ b-a}{2^{j}},b- \frac{b-a}{2^{j+1}} \right),\quad  j=1,2,\ldots,
\end{equation*}
We denote by $\{I_k\}$ the family of all these intervals. It gives a Withney partition of the set $G$. Observe that each $I_k$ has two adjacent intervals $I_k^+$ and $I_k^-$. We denote
\begin{equation*}
I_k^*=I_k\cup I_k^+\cup I_k^-.
\end{equation*}
We have
\begin{align}
&G=\bigcup_{k=1}^\infty I_k,\label{b31}\\
&\dist(I_k,G^c)=|I_k|.\label{b32}\\
&\dist(I_j,I_k)\ge |I_j|/2,\text { if } I_j\cap I_k^*=\varnothing.\label{b37}
\end{align}
In the proof of the next lemma we use Stein-Weiss \cite{StWe} well known identity. That is for any set $E\subset \ZR$ of finite measure we have
\begin{equation*}
|\{x\in\ZR:\, |H\ZI_E(x)|>\lambda\}|= \frac{4e^{\pi\lambda}|E|}{e^{2\pi\lambda}-1},\quad \lambda>0.
\end{equation*}
\begin{lemma}\label{L2}
Let $G$ be an open set with Withney partition $\{I_k\}$ and let $e\subset G$ be a null set. Then for any $\gamma>0$ and a sequence of numbers $\delta_k>0$ there exists an open set $F$ with $e\subset F\subset G$ such that
\begin{align}
&\{x\in \ZR:\, |H\ZI_F(x)|>\gamma\}\subset G,\label{z2}\\
&|I_k\cap \{x\in \ZR:\, |H\ZI_F(x)|>\gamma\}|\le \delta_k,\quad k=1,2,\ldots .\label{b47}
	\end{align}
If the set $e$ additionally is compact, then $F$ can be taken to be finite.
\end{lemma}
\begin{proof}

We define $F$ to be an open set satisfying
\begin{equation}\label{b20}
|F\cap I_j^*|<\min\left\{\frac{\pi\gamma|I_j|}{2^{j+2}},\frac{\delta_j(e^{\pi\gamma}-1)}{4e^{\pi\gamma/2}}\right\},\quad j=1,2,\ldots.
\end{equation}
If $e$ is compact, then clearly $F$ can be finite. If $x\in \ZR\setminus G$, then $\dist (x,I_j)\ge |I_j|$. Thus we get
\begin{align*}
|H\ZI_F(x)|&\le \left|\sum_{j=1}^\infty H\ZI_{F\cap I_j} (x)\right|\le \frac{1}{\pi}\sum_{j=1}^\infty\left|\int_{F\cap I_j}\frac{dt}{x-t}\right|\\
&\le \frac{1}{\pi}\sum_{j=1}^\infty\frac{|F\cap I_j|}{\dist(x,I_j)}\le  \frac{1}{\pi}\sum_{j=1}^\infty\frac{|F\cap I_j|}{|I_j|}<\gamma
\end{align*}
and so \e {z2}.
Then, using \e{b37}, for any $x\in I_k$ we get
\begin{align}\label{b36}
\left|\sum_{j:\,I_j\cap I_k^*=\varnothing}H\ZI_{F\cap I_j} (x)\right|&\le \frac{1}{\pi}\sum_{j:\,I_j\cap I_k^*=\varnothing}\left|\int_{F\cap I_j}\frac{dt}{x-t}\right|\\
&\le \frac{1}{\pi}\sum_{j:\,I_j\cap I_k^*=\varnothing}\frac{|F\cap I_j|}{\dist(I_j,I_k)}\nonumber\\
&\le \frac{1}{\pi} \sum_{j=1}^\infty\frac{2|F\cap I_j|}{|I_j|}<\frac{\gamma}{2}.\nonumber
\end{align}
If $I_j\cap I_k^*\neq\varnothing$, then $I_j$ coincides with one of the intervals $I_k$, $I_k^+$ or $I_k^-$. Thus by the Stein-Weiss inequality we get
\begin{align}\label{b35}
\bigg|\bigg\{x\in \ZR:\, \bigg|\sum_{j:\,I_j\cap I_k^*\neq\varnothing} &H\ZI_{F\cap I_j}(x)\bigg|>\gamma/2\bigg\}\bigg|\\
&=\bigg|\bigg\{x\in \ZR:\, \bigg|H\ZI_{F\cap I_k^*}(x)\bigg|>\gamma/2\bigg\}\bigg|\\
&\le\frac{4e^{\pi\gamma/2}}{e^{\pi\gamma}-1}|F\cap I_k^*|\le\delta_k. \nonumber
\end{align}
From \e{b36} and \e{b35} we obtain
\begin{align*}
&|\{x\in I_k:\, |H\ZI_F(x)|>\gamma\}|\\
&\qquad\le\left|\left\{x\in I_k:\,\left|\sum_{j:\,I_j\cap I_k^*=\varnothing}H\ZI_{F\cap I_j} (x)\right|>\gamma/2\right\}\right|\\
&\qquad\quad +\left|\left\{x\in I_k:\,\left|\sum_{j:\,I_j\cap I_k^*\neq\varnothing}H\ZI_{F\cap I_j} (x)\right|> \gamma/2\right\}\right|\\
&\qquad\le\left\{x\in \ZR:\,\left|\sum_{j:\,I_j\cap I_k^*\neq\varnothing}H\ZI_{F\cap I_j} (x)\right|>\gamma/2\right\}\le \delta_k.
\end{align*}
\end{proof}
\begin{lemma}\label{L3}
Let $G$ be an open set and  $e\subset G$ be a null set. Then for any $\delta>0$ and $\lambda>0$, $\mu<0$, satisfying \e {z1}, there exists an open set $F$ such that
\begin{align}
&e\subset F\subset G,\label{b18}\\
&E=\{x\in \ZR\setminus F:\, H\ZI_F(x)<\mu\}\subset G\label{z3}\\
&F\subset \{x\in\ZR\setminus F:\,H\ZI_E(x)> \lambda\},\label{b26}
\end{align}
and
\begin{equation}
\frac{1}{\pi}\int_{|t-x|>\varepsilon} \frac{\ZI_E(t)}{|x-t|}dt<\delta,\label{b19}
\end{equation}
whenever 
\begin{equation}\label{b21}
(x-\varepsilon,x+\varepsilon)\not\subset G.
\end{equation}
If $e$ is compact, then $F$ is finite-open.
\end{lemma}
\begin{proof}
Let $I_k=[a_k,b_k)$ be a Withney partition of $G$ defined above. Applying \lem{L2}, we find an open set $F$ satisfying  \e {b18}, \e {z2} and \e{b47} for the numbers
\begin{equation*}
\gamma=|\mu|,\quad \delta_k=\frac{\pi\delta|I_k|}{2^{k+1}}.
\end{equation*}
Such that
\begin{equation*}
E\subset \{x\in \ZR\setminus F:\, |H\ZI_F(x)|>\gamma\},
\end{equation*}
from \e {z2} implies \e {b18}.
Since $F$ is open, from \trm {T1} it follows that
\begin{equation*}
F\subset F^*=\{x\in \ZR\setminus E:\, \H\ZI_E(x)> \lambda \},
\end{equation*}
which implies \e {b26}. From \e{b47} we get
\begin{equation*}
|E\cap I_k|\le \delta_k=\frac{\delta|I_k|}{2^{k+1}}.
\end{equation*}
Take an arbitrary $x\in \ZR$ and $\varepsilon>0$ satisfying \e{b21}. We claim that
\begin{equation}\label{b25}
a_k(x)=\frac{1}{\pi}\int_{|t-x|>\varepsilon} \frac{\ZI_{E\cap I_k}(t)}{|x-t|}dt\le \frac{\delta}{2^{k}}.
\end{equation}
Indeed, if
\begin{equation*}
 I_k\cap(x-\varepsilon,x+\varepsilon)=\varnothing,
\end{equation*}
then from \e{b21} and \e{b32} one can easily get  $\dist(x,I_k)\ge|I_k|/2$ and therefore 
\begin{equation*}
a_k(x)\le \frac{|E\cap I_k|}{\pi\dist(x,I_k)}\le \frac{2|E\cap I_k|}{\pi|I_k|}\le \frac{\delta}{2^{k}}.
\end{equation*}
In the case
\begin{equation*}
 I_k\cap(x-\varepsilon,x+\varepsilon)\neq\varnothing,
\end{equation*}
again taking into account of \e{b21}, we get $|I_k|\le 2\varepsilon$. Then the bound
\begin{equation*}
a_k(x)\le \frac{|E\cap I_k|}{\pi\varepsilon} \le  \frac{|E\cap I_k|}{\pi|I_k|/2} \le \frac{\delta}{2^{k}}
\end{equation*}
establishes \e{b25}. Thus, applying \e {b21}, for $x$ satisfying \e {b21} we get
\begin{equation*}
\left|H_\varepsilon\ZI_E(x)\right|\le \sum_{k=1}^\infty \left|H_\varepsilon\ZI_{E\cap I_k}(x)\right|\le \sum_{k=1}^\infty  \frac{\delta}{2^{k}} \le \delta.
\end{equation*}\
\end{proof}
\begin{lemma}\label{L4}
	There exists a function $\varphi \in C(\ZR)$ with $\supp \varphi\subset [-1,1]$, such that $H\varphi(x)$ is finitely defined for any $x\neq 0$ and $H^*\varphi(0)=\infty$.
\end{lemma}
\begin{proof}
 Define $\varphi(x)$ by
		\begin{align}
		\varphi(x)=\left\{
		\begin{array}{llr}
		1-x &\hbox{ if } x\in [0,1 ],\\
		1-1/(k+1)&\hbox { if } x=-2^{-k},\, k=0,1,\ldots ,\\
		\hbox{linear on each interval } &[-2^{-k},-2^{-k-1}),\, k=0,1,\ldots .\\
		\end{array}
		\right.
		\end{align}
One can check that $f$ is continuous. Linearity implies the existence of $H\varphi(x)$ for any $x\neq 0$. Then we have
\begin{align*}
H_{\delta/2^n}\varphi(0)&\le \int_{2^{-n}}^1\frac{(1-t)}{-t}dt+
\sum_{k=1}^n\int_{-2^{-k+1}}^{-2^{-k}}\frac{(1-1/(k+1))}{-t}dx\\
&=-n\ln 2+(1-2^{-n})+\sum_{k=1}^n\left(1-\frac{1}{k+1}\right)\ln 2\\
&=-\ln 2\sum_{k=1}^n\frac{1}{k+1}+1-2^{-n}
\end{align*} 
that means $H^*\varphi(0)=\infty$.
\end{proof}

\begin{proof}[Proof of \trm {T1}]
Let $e\subset \ZR$ be a set of measure zero. Applying \lem{L3} successively (with $\lambda=1$, $\mu=\pi^{-1}\ln (1-e^{-1})$), we find a sequences of open sets $F_n$ such  that
\begin{align}
&F_{n-1}\supset F_{n}\supset e,\label{b50} \\
&E_n=\{x\in \ZR\setminus F_n:\, H\ZI_{F_n}(x)<\mu\}\subset F_{n-1}\setminus F_n,\label{z4}\\
&F_n\subset \{x\in\ZR\setminus E_n:\,H\ZI_{E_n}(x)> 1\},\label{b51}\\
&\int_{|t-x|>\varepsilon} \frac{\ZI_{E_n}(t)}{|x-t|}dt<2^{-n},\text{ if } (x-\varepsilon,x+\varepsilon)\not\subset F_{n-1}.\label{b52}
\end{align} 
Define 
\begin{equation*}
E=\bigcup_{n=1}^\infty E_n.
\end{equation*}
Take an arbitrary point $x\in e$. Since the sets $E_n$ are pairwise disjoint, we can write
\begin{align}
H_\varepsilon \ZI_{E}(x)&=\sum_{n=1}^\infty H_\varepsilon \ZI_{E_n}(x)\label{b53}\\
&=\sum_{(x-\varepsilon,x+\varepsilon)\not\subset F_{n-1}} H_\varepsilon \ZI_{E_n}(x)+\sum_{(x-\varepsilon,x+\varepsilon)\subset F_{n-1}} H_\varepsilon \ZI_{E_n}(x)\nonumber\\
&=A+B.\nonumber
\end{align}
Using \e{b52}, we get
\begin{equation}\label{b54}
|A|\le \sum_{(x-\varepsilon,x+\varepsilon)\not\subset F_{n-1}}\int_{|t-x|>\varepsilon} \frac{\ZI_{E_n}(t)}{|x-t|}dt <\sum_{n=1}^\infty 2^{-n}=1.
\end{equation}
From \e {b50} and \e{b51} we obtain
\begin{equation}\label{b55}
B=\sum_{(x-\varepsilon,x+\varepsilon)\subset F_{n-1}} H\ZI_{E_n}(x)\ge \sum_{(x-\varepsilon,x+\varepsilon)\subset F_{n-1}}1.
\end{equation}
The number of terms in the last sum can be arbitrarily big, if we take $\varepsilon >0$ sufficiently small. So combining \e{b53}-\e{b55}, we get
\begin{equation*}
H^*\ZI_{E}(x)=\infty.
\end{equation*}
\end{proof}

\begin{proof}[Proof of \trm {T2}]
Let us suppose first that $e$ is a compact null set and we have $E\subset [a,b]$. Applying \lem{L3} successively (with $\lambda=2^n$, $\mu=\frac{1}{\pi}\ln(1-e^{-2^n})$ we find a sequences of finite-open sets of the form
\begin{equation*}
F_n=\bigcup_{k=1}^{m_n}(a_k^{(n)},b_k^{(n)}),
\end{equation*}
such  that
\begin{align}
&F_{n-1}\supset F_{n}\supset e,\label{g50} \\
&E_n=\{x\in \ZR\setminus F_n:\, H\ZI_{F_n}(x)<\mu \}\subset F_{n-1}\setminus F_n,\label{z5}\\
&F_n= \{x\in\ZR\setminus E_n:\,H\ZI_{E_n}(x)> \lambda =2^n\}\label{g51}\\
&\int_{|t-x|>\varepsilon} \frac{\ZI_{E_n}(t)}{|x-t|}dt<1\text{ if } (x-\varepsilon,x+\varepsilon)\not\subset F_{n-1}.\label{g52}
\end{align} 
From the finiteness of the open sets $F_n$ we get the finiteness of $E_n$. Thus the equality in \e {g51} will follow from \trm {T3}. From \e{g51} it follows that
\begin{equation}\label{z6}
H\ZI_{E_n}(b_k^{(n)})=2^n.
\end{equation}
Observe that one can choose functions $f_n\in C(\ZR)$ such that
\begin{align}
&0\le f_n\le  \ZI_{E_n},\label{z7}\\
&|Hf_n(b_k^{(n)})-H\ZI_{E_n}(b_k^{(n)})|<1,\quad k=1,2,\ldots,m_n.\label{z8}
\end{align}
Thus, taking into account \e {z6}, we get $Hf_n(b_k^{(n)})>2^n-1$. Then, applying the monotonicity property, we conclude
\begin{equation}\label{z9}
Hf_n(x)>2^n-1,\quad x\in F_n.
\end{equation}
Define
\begin{equation*}
f(x)=\sum_{n=1}^\infty \frac{f_n(x)}{2^n}.
\end{equation*}
One can easily check that $f\in C(\ZR)$. Take an arbitrary point $x\in e$. Using \e {z7}, we can write
\begin{align}
H_\varepsilon f(x)&=\sum_{n=1}^\infty \frac{H_\varepsilon f_n(x)}{2^n}\label{g53}\\
&=\sum_{(x-\varepsilon,x+\varepsilon)\not\subset F_{n-1}} \frac{H_\varepsilon f_n(x)}{2^n}+\sum_{(x-\varepsilon,x+\varepsilon)\subset F_{n-1}} \frac{H_\varepsilon f_n(x)}{2^n}\nonumber\\
&=A+B.\nonumber
\end{align}
Applying \e{g52}, we get
\begin{align}
|A|&\le \sum_{(x-\varepsilon,x+\varepsilon)\not\subset F_{n-1}}\frac{1}{2^n}\int_{|t-x|>\varepsilon} \frac{f_n(t)}{|x-t|}dt\label{g54}\\
&\le \sum_{(x-\varepsilon,x+\varepsilon)\not\subset F_{n-1}}\frac{1}{2^n}\int_{|t-x|>\varepsilon} \frac{\ZI_{E_n}(t)}{|x-t|}dt<\sum_{n=1}^\infty 2^{-n}=1.\nonumber
\end{align}
From \e {z7} and \e {z5} it follows that $\supp f_n\subset F_{n-1}$. Thus, using \e {z9}, we obtain
\begin{equation}\label{g55}
B=\sum_{(x-\varepsilon,x+\varepsilon)\subset F_{n-1}} \frac{Hf_n(x)}{2^n}\ge \sum_{(x-\varepsilon,x+\varepsilon)\subset F_{n-1}}(1-2^{-n}).
\end{equation}
For sufficiently small $\varepsilon >0$ the last sum can be arbitrarily big. So combining \e{g53}-\e{g55}, we get
\begin{equation*}
H^*f(x)=\infty.
\end{equation*}
Now suppose $e\subset \ZR$ is an arbitrary closed set. Take a sequence $x_k\in e$, $k\in \ZZ$,  such that 
\begin{equation*}
x_{k+1}-x_k>1,\quad e=\bigcup_k e\cap [x_k,x_{k+1}]
\end{equation*}
and consider the compact sets $e_k=e\cap [x_k,x_{k+1}]$. For each $k$ we can find a continuous functions $f_k$ such that
\begin{equation*}
H^*f_k(x)=\infty,\quad x\in e_k.
\end{equation*}
Denote 
\begin{equation*}
g(x)=\sum_k \frac{1}{2^k(x_{k+1}-x_k)}f_k(x)\ZI_{[x_k,x_{k+1}]}(x)\in L^1(\ZR).
\end{equation*}
One can easily check that
\begin{equation*}
H^*g(x)=\infty, \quad x\in e\setminus \{x_k\}.
\end{equation*}
Set 
\begin{equation*}
\varepsilon_k=\left\{\begin{array}{lcl}
0&\hbox{ if }& H^*g(x_k)=\infty,\\
1&\hbox{ if }& H^*g(x_k)<\infty.
\end{array}
\right.
\end{equation*}
Define
\begin{equation*}
f(x)=g(x)+\sum_k\varepsilon_k 2^{-k}\lambda(x-x_k),
\end{equation*}
where $\varphi$ satisfies the conditions of \lem {L4}. It is clear that the function satisfies the conditions of \trm {T2}.
\end{proof}
\section{Remark on Kahane-Katznelson divergence theorem}
Kahane-Katznelson constructed a complex valued continuous function whose Fourier series diverges on a given set of zero measure. The proof of this theorem is based on the following lemma, which was proved by methods of analytic functions. We will deduce this lemma from \lem {L1}. We shall use also the following well-known relation between two functions $f\in L^1(\ZT)$, $g\in L^\infty(\ZT)$ (see. \cite{Zyg}, ch. 2, Theorem 4.15) 
\begin{equation}\label{z53}
\lim_{n\to\infty}\int_0^{2\pi}f(x)g(nx)dx=\frac{1}{2\pi}\int_0^{2\pi}f(x)dx\cdot\int_0^{2\pi}g(x)dx.
\end{equation}
\begin{lemma}[Kahane-Katznelson]
If $F\subset \ZT$ is a finite-open set $|F|= \alpha>0$, then there exists a complex trigonometric polynomial $P(x)$ of degree $n$ such that
\begin{equation}\label{z11}
\max_{1\le m\le n}|S_m(x,P)|\ge c \ln \frac{1}{\alpha},\quad x\in F.
\end{equation}
\end{lemma}
\begin{proof}
	Let $F$ has the form \e {b7} and first suppose that 
	\begin{equation}\label{z10}
	\max F-\min F\le \pi. 
	\end{equation}
	Without loss of generality we can suppose $F\subset [0,\pi]$.
	For $\lambda=\ln \frac{\pi}{\alpha}$ we apply \lem {L1}. We get an open set $E$ of the form \e {b44} satisfying the conditions of lemma. From \e {b41} it follows that
	\begin{equation*}
	|E|=\left(\frac{\pi}{\alpha}-1\right)|F|=\pi-\alpha<\pi.
	\end{equation*}
	According the structure of the set $E$ coming from \lem {L1}, we get $E\subset [-\pi,\pi]$. Take $\delta>0$ and consider the following modification of the set $E$: 
	\begin{equation}\label{z55}
	\tilde E=\bigcup_{k=1}^n(c_k+\delta,a_k-\delta).
	\end{equation}
	For small enough $\delta$ from \e {b40} we get
	\begin{equation*}
	H\ZI_{\tilde E}(x)>\lambda,\quad x\in F.
	\end{equation*}
	Set
	\begin{equation*}
	f_m(x)=\ZI_{\tilde E}(x)\sign(\sin mx),\quad m=1,2,\ldots.
	\end{equation*}
	For the modified partial sums we have 
	\begin{align*}
	S^*_m(x,f_m)&=\frac{1}{\pi}\int_{-\pi}^{\pi}\frac{\sin m(x-t)}{x-t}f_m(t)dt\\
	&=\frac{\sin mx}{\pi}\int_{-\pi}^{\pi}\frac{\cos mt}{x-t}f_m(t)dt-\frac{\cos mx}{\pi}\int_{-\pi}^{\pi}\frac{\sin mt}{x-t}f_m(t)dt
	\end{align*}
Applying \e {z53}, as $m\to\infty $ for $x\in F$ we get 
\begin{align*}
\int_{-\pi}^{\pi}\frac{\cos mt}{x-t}f_m(t)dt&=\int_{-\pi}^{\pi}\frac{\ZI_{\tilde E}(t)}{x-t}\cdot \cos mt\,\sign(\sin mt)dt\\
&\to \frac{1}{2\pi}\int_{-\pi}^{\pi}\frac{\ZI_{\tilde E}(t)}{x-t}dt\cdot \int_{-\pi}^{\pi}\cos t\,\sign(\sin t)dt=0
\end{align*}
and 
\begin{align*}
\int_{-\pi}^{\pi}\frac{\sin mt}{x-t}f_m(t)dt&=\int_{-\pi}^{\pi}\frac{\ZI_{\tilde E}(t)}{x-t}\cdot \sin mt\,\sign(\sin mt)dt\\
&\to \frac{1}{2\pi} \int_{-\pi}^{\pi}\frac{\ZI_{\tilde E}(t)}{x-t}dt\cdot \int_{-\pi}^{\pi}|\sin t|dt=\frac{\pi}{2}\cdot H\ZI_{\tilde E}(x).
\end{align*}
Since the sets $\tilde E$ and $F$ have positive distance (see \e {z55}, \e {b7}), in both limits the convergence is uniformly on $F$. Thus for enough bigger $m$ we will have
\begin{equation*}
|S^*_m(x,f_m)|>\frac{\pi\lambda }{2}|\sin mx|,\quad x\in F.
\end{equation*}
Again, since $\dist(\tilde E,F)>0$, a proper approximation of $f_m$ by a real polynomial $f(x)$ implies
\begin{equation*}
|S^*_m(x,f)|>\frac{\pi\lambda }{3}|\sin mx|,\quad x\in F.
\end{equation*}
Similarly, taking instead of $f_m$ the sequence 
\begin{equation*}
g_m(x)=\ZI_{E}(x)\sign(\cos mx),\quad m=1,2,\ldots,
\end{equation*}
we will get another real polynomial $g(x)$ such that 
\begin{equation*}
|S^*_m(x,g)|>\frac{\pi\lambda }{3}|\cos mx|,\quad x\in F.
\end{equation*}
For the complex polynomial $P=f+gi$ we will have
\begin{equation*}
|S^*_m(x,P)|\ge \frac{\pi\lambda }{3}(|\sin mx|+|\cos mx|)\ge \frac{\pi\lambda }{3},\quad x\in F,
\end{equation*}
and so it will satisfy the condition of lemma. If $F$ is arbitrary, then we have $F=F_1\cup F_2$, where each $F_1$ and $F_2$ satisfy \e {z10}. Let $P_1$ and $P_2$ be the polynomials corresponding to $F_1$ and $F_2$. Suppose the degrees of those polynomials are less that $n$. One can check that the polynomial $P(x)=P_1(x)+e^{inx}P_2(x)$ satisfies \e {z11}.
\end{proof}

\end{document}